\documentclass{amsart}

\usepackage{amsmath,amssymb}
\usepackage{graphicx}
\usepackage[usenames,dvipsnames,svgnames,table]{xcolor}
\usepackage[pdftex,breaklinks=true]{hyperref}
\usepackage{multirow}
\usepackage{multicol}
\usepackage{smartref}

\allowdisplaybreaks

\newtheorem{theorem}{Theorem}[section]

\newtheorem{lemma}[theorem]{Lemma}

\newtheorem{remark}[theorem]{Remark}

\theoremstyle{definition}
\newtheorem{definition}[theorem]{Definition}

\theoremstyle{remark}

    \renewenvironment{proof}[1][Proof]{ \textbf{#1.} }{\qed\\}

\newcommand{\parc}[1]{\partial_{#1}}

\newcommand{\R}{\mathbb{R}}     
\newcommand{\C}{\mathbb{C}}     

\numberwithin{equation}{section}

\DeclareMathOperator{\sign}{sign}
\DeclareMathOperator{\Div}{div}

\begin{document}

\title[Realization problems for limit cycles]{Realization problems for limit cycles of planar polynomial vector fields}

\author{Juan Margalef-Bentabol}
\address{Instituto de Estructura de la Materia, Consejo Superior de  Investigaciones Cient\'ificas, 28006 Madrid, Spain, and Instituto Gregorio Mill\'an, Grupo de Modelizaci\'on y Simulaci\'on Num\'erica, Universidad Carlos III de Madrid, 28911 Legan\'es, Spain}
\email{juan.margalef@uc3m.es}

\author{Daniel Peralta-Salas}
\address{Instituto de Ciencias Matem\'aticas, Consejo Superior de  Investigaciones Cient\'\i ficas, 28049 Madrid, Spain}
\email{dperalta@icmat.es}

\date{\today}

\begin{abstract}
We show that for any finite configuration of closed curves $\Gamma\subset \R^2$, one can construct an explicit planar polynomial vector field that realizes $\Gamma$, up to homeomorphism, as the set of its limit cycles with prescribed periods, multiplicities and stabilities. The only obstruction given on this data is the obvious compatibility relation between the stabilities and the parity of the multiplicities. The constructed vector fields are Darboux integrable and admit a polynomial inverse integrating factor.
\end{abstract}
\maketitle

\section{Introduction and statement of the main theorem}

We consider the planar vector field
\begin{equation}\label{eq1}
X=P(x,y)\parc{x}+Q(x,y)\parc{y}\,,
\end{equation}
where $P$ and $Q$ are polynomials. The degree of $X$ is defined as the maximum of the degrees of $P$ and $Q$, and will be denoted by $\deg(X)$. The study of the limit cycles of planar polynomial vector fields has a long tradition, starting with the celebrated second part of Hilbert's 16th problem, cf.~\cite{Ilyashenko 2002}, which consists in finding the maximum number of limit cycles for the vector field~\eqref{eq1} in terms of the degree $\deg(X)$, and studying the relative positions of these cycles. We recall that a limit cycle of $X$ is a periodic trajectory that is isolated. Despite having been formulated more than a century ago, Hilbert's 16th problem is still wide open, even for quadratic vector fields, i.e. $\deg(X)=2$.

A related problem that has attracted considerable attention in the last years is the realization problem for limit cycles. To state it in a precise way let us introduce some basic definitions. Let $C$ be a closed curve embedded in $\R^2$. A configuration of cycles is a finite set $\Gamma=\{C_1,\ldots,C_n\}$ of closed curves.
\begin{definition} \label{def_C_r_equivalentes}
We say that two configurations of cycles $\Gamma,\Gamma'$ are \emph{equivalent} if there exists a homeomorphism $\phi:\R^2\rightarrow\R^2$ such that $\phi(\Gamma)=\Gamma'$. A vector field $X$ is said to \emph{realize} a configuration of cycles $\Gamma$ if its set of limit cycles is equivalent to $\Gamma$.
\end{definition}

The realization (or inverse) problem asks if, for any configuration of cycles $\Gamma$, there exists a planar vector field $X$ that realizes $\Gamma$. In this problem it is usual to prescribe other dynamical properties of the limit cycles (e.g. stability) as well as some additional conditions on the vector field $X$ (e.g. regularity).

The realization problem for limit cycles was first addressed by Al'mukhamedov~\cite{Al63} for $C^k$ vector fields using the theory of Lyapunov functions. The same techniques allowed Sverdlove~\cite{Sverd81} to solve the $C^k$ realization problem completely (also prescribing the stability of the cycles), fixing a gap in Al'mukhamedov's construction that had been noticed by other authors, see~\cite{Sverd81} for details.

Regarding polynomial vector fields, the first realization result was obtained by Bautin in~\cite{Ba80} (later corrected in~\cite{DK93}), where explicit expressions for polynomial planar vector fields with a prescribed set of algebraic limit cycles were derived. Alternative constructions appear in the works of Winkel~\cite{Wi00}, Christopher~\cite{Chris01} and Korchagin~\cite{Ko05}. The idea of using inverse integrating factors to construct polynomial vector fields with prescribed algebraic limit cycles was first given in~\cite{G96,CGG97}. This idea culminated in the complete solution to the realization problem in the polynomial setting obtained by Llibre and Rodr\'iguez in~\cite{Llibre Rodriguez 2004} using inverse integrating factors and the Darboux theory of integrability. An alternative proof using Lyapunov functions was given in~\cite{Peralta 2005}, with a generalization to higher dimensions. A remarkable recent solution of the realization problem was obtained by Coll, Dumortier 
 and Prohens~\cite{CDP13} using polynomial Li\'enard equations and the theory of slow-fast systems. All these works provide an upper bound for $\deg(X)$ which depends on the configuration of cycles. Notice that since the degree of $X$ is not fixed a priori, the realization problem is much easier than Hilbert's 16th problem.

The main result of this paper is a realization theorem for planar polynomial vector fields prescribing not only the configuration of limit cycles, but also their periods, multiplicities and stabilities. These are the three most basic invariants under smooth conjugacy, so it is reasonable to consider them as a part of the realization problem. The limit cycles of the systems constructed in~\cite{Llibre Rodriguez 2004} and~\cite{CDP13} are all hyperbolic and their stabilities depend on the configuration that is realized (of course, there are no semistable limit cycles due to the hiperbolicity).

To state our main theorem let us introduce some notation. We call $D_C$ the compact set bounded in $\R^2$ by the closed curve $C$. We say that a limit cycle $C$ is stable (unstable) in the interior if all the trajectories of $X$ in $D_C$ that are sufficiently close to $C$ approach the limit cycle as $t\to \infty$ ($t\to -\infty$). Analogously, if all the trajectories of $X$ in $\R^2\setminus D_C$ that are close enough to $C$ approach the limit cycle as $t\to \infty$ ($t\to -\infty$), we say that $C$ is stable (unstable) in the exterior. It is obvious that the exterior stability of a cycle is determined by its interior stability and multiplicity.

\begin{theorem}\label{T:main}
Consider sets $\{T_1,\ldots,T_n\}$ of positive constants and $\{m_1,\ldots,m_n\}$ of positive integers. Let $\Gamma=\{C_1,\ldots,C_n\}$ be a configuration of cycles with fixed interior stabilities. Then $\Gamma$ is realized as the set of limit cycles of a polynomial vector field $X$, where each $C_k$ has multiplicity $m_k$, period $T_k$ and the required stability. Moreover, there exists an explicit upper bound for the degree of $X$ in terms of the prescribed quantities (see Theorem~\ref{main theorem}).
\end{theorem}

The proof of Theorem~\ref{T:main} is based on the construction of Llibre and Rodr\'iguez~\cite{Llibre Rodriguez 2004}. As in~\cite{Llibre Rodriguez 2004}, the realized limit cycles are algebraic and the vector fields are Darboux integrable and admit a polynomial inverse integrating factor. We would like to remark that the limit cycles of a polynomial vector field do not need to be algebraic, cf.~\cite{Odani 1995}, but we are not aware of any construction of polynomial vector fields with given limit cycles that are not algebraic. It thus remains an interesting open problem to characterize those analytic curves that can be limit cycles of a polynomial vector field.

The structure of the paper is as follows. In Section~\ref{section periods} we construct a vector field $X_T$ realizing the configuration of cycles $\Gamma$ with prescribed periods. The construction of a vector field $X_m$ with prescribed limit cycles and multiplicities is presented in Section~\ref{section multiplicity}, where we also combine both constructions to obtain a vector field $X_{Tm}$ prescribing periods and multiplicities at the same time. In Section~\ref{section stability} we construct a vector field $X_{Ts}$ realizing the configuration of cycles $\Gamma$ with prescribed periods and stabilities. Finally, using the constructions in the previous sections, the main theorem is proved in Section~\ref{section main theorem}.

\section{A planar polynomial vector field with prescribed limit cycles and periods}\label{section periods}
In the following theorem we show that any configuration of cycles $\Gamma$ can be realized by a planar polynomial vector field, where the period of each limit cycle is also prescribed. The method of the proof is a variation of the construction introduced in~\cite{Llibre Rodriguez 2004}. In the statement of the theorem, a curve $C_k\in\Gamma$ is called primary if no other curve $C_{j}\in\Gamma$, $j\neq k$, is contained in the domain $D_{C_k}$.

We recall that a \emph{Darboux first integral} is a (possibly multivalued) function $G$ that can be written as:
\[G=e^{g/h}\prod_ {l=1}^Lf_l^{\lambda_l}\,,\]
where $f_l$, $g$ and $h$ are complex polynomials and $\lambda_l\in\C$ are complex constants, and
a smooth function $V$ is an \emph{inverse integrating factor} of a vector field $X$ if
\[X\cdot \nabla V=V \Div(X)\,,\]
where $\Div$ denotes the divergence operator.

\begin{theorem}\label{teorema period}
Let $\Gamma=\{C_1,\ldots,C_n\}$ be a configuration of cycles and $\{T_1,\ldots,T_n\}$ a set of positive constants, then $\Gamma$ is realized by a planar polynomial vector field $X_T$ with $\deg(X_T)<2(n+r)$, where each periodic orbit $C_k$ has period $T_k$. Here $r$ is the number of primary cycles in $\Gamma$. Moreover, $X_T$ admits a polynomial inverse integrating factor, it is Darboux integrable and all its limit cycles are algebraic and hyperbolic.
\end{theorem}
\begin{proof}
It is well known~\cite{Llibre Rodriguez 2004} that any configuration of cycles is homeomorphic to a set of circles. Accordingly, we can take that $\Gamma$ consists of $n$ disjoint circles of radii $\{r_k\}_{k=1}^n$ centered at the points $\{p_k\}_{k=1}^n$, i.e.
\begin{equation}\label{eq:circles}
C_k=\{f_k(x,y)=0\} \quad \text{ with } \quad f_k(x,y):=(x-x_k)^2+(y-y_k)^2-r^2_k\,,
\end{equation}
where $p_k=(x_k,y_k)$. We can safely assume that the first $r$ cycles $\{C_1,\ldots,C_r\}$ are primary and that no point $p_k$ is contained in any circle $C_j$.

Let us introduce the following auxiliary functions:
\begin{multicols}{2}\noindent
\begin{align*}
&g_k(x,y):=(x-x_k)^2+(y-y_k)^2\,,\\
&A:=\prod_{k=1}^n f_k\,,\\
&A_T:=\prod_{k=1}^n f_k^{\tau_k}\,,\\
&B:=\prod_{k=1}^r g_k\,,\\
&C:=\exp\left(-2\sum_{k=1}^r\theta_k\right)\,,\\
&D_T:=A_TBC\,,\\
&H_T:=\ln{D_T}\,,
\end{align*}
\end{multicols}
\noindent where $\tau_k$ are positive constants that will be fixed later in order to prescribe the desired periods $T_k$, and $\theta_k$ is an angular (multivalued) function defined as
\begin{equation}\label{thetak}
\theta_k:=\arctan\left(\frac{y-y_k}{x-x_k}\right)\,.
\end{equation}
We claim that the vector field $X_T$ in the statement of the theorem can be defined as
\[
X_T:=P_T\partial_x+Q_T\partial_y\,, \qquad\text{with}\qquad \begin{array}{l}
P_T:=-AB\dfrac{\partial H_T}{\partial y}\,,\\[2ex]
Q_T:= AB\dfrac{\partial H_T}{\partial x}\,.
\end{array}\]

We observe that in the case that $\tau_k=1$ for all $k$, then $A_T=A$ and the vector field $X_T$ is the same as in the construction by Llibre and Rodr\'{i}guez~\cite{Llibre Rodriguez 2004}. Throughout the paper, we shall use the notation $X_{LR}:=X_T|_{\tau_k=1}$.

To prove that $X_T$ satisfies all the claims in the statement of the theorem, let us first show that it is a polynomial vector field with $\deg(X_T)<2(n+r)$, that $D_T$ is a Darboux first integral and that $V_T:=AB$ is an inverse integrating factor. Indeed, noticing that
\begin{equation}\label{idC}
B\frac{\partial C}{\partial x}=C\frac{\partial B}{\partial y}\,,\qquad B\frac{\partial C}{\partial y}=-C\frac{\partial B}{\partial x}\,,
\end{equation}
a straightforward computation shows that $P_T$ and $Q_T$ can be written as
\begin{align}
P_T=A\left(\frac{\partial B}{\partial x}-\frac{\partial B}{\partial y}\right)-B\sum_{k=1}^n \tau_k\mu_k\frac{\partial f_k}{\partial y}\,,\label{eqpt}\\
Q_T=A\left(\frac{\partial B}{\partial x}+\frac{\partial B}{\partial y}\right)+B\sum_{k=1}^n \tau_k\mu_k\frac{\partial f_k}{\partial x}\,,\label{eqqt}
\end{align}
where 
$$\mu_k:=\prod_{j\neq k}^n f_j\,.$$ 
It is clear from these expressions that $P_T$ and $Q_T$ are polynomials of degree at most $2(n+r)-1$. Moreover, it is easy to check that the Darboux function $D_T$ satisfies $X_T\cdot\nabla D_T=0$, thus implying that $D_T$ is a Darboux first integral of $X_T$. Another easy computation shows that the vector field $\frac{X_T}{V_T}$ is divergence-free in $\mathbb R^2\setminus V_T^{-1}(0)$, and hence $V_T$ is an inverse integrating factor of $X_T$.

Accordingly, all the limit cycles of $X_T$ are contained in the zero set $V_T^{-1}(0)$ of its inverse integrating factor~\cite{GLV96}. Therefore, since
\[V_T^{-1}(0)=\Gamma\cup \{p_1,\ldots,p_r\}\,,\]
we conclude that if $X_T$ has a limit cycle, it has to be precisely one of the circles $\{C_1,\ldots,C_n\}$. Let us now prove that indeed all of them are realized as limit cycles.

We first show that each $C_k$ is a periodic trajectory of $X_T$. Since $V_T^{-1}(0)$ is invariant under the flow of $X_T$, it is enough to prove that $X_T$ does not vanish on each $C_k$. Using Eqs.~\eqref{eqpt} and~\eqref{eqqt} we can evaluate the components $P_T$ and $Q_T$ of the vector field on each circle $C_k=\{f_k=0\}$, thus obtaining
\begin{align}\label{eqxTck}
X_T|_{C_k}=\tau_kB|_{C_k}\mu_k|_{C_k}\left(-\frac{\partial f_k}{\partial y}\Big|_{C_k}\partial_x+\frac{\partial f_k}{\partial x}\Big|_{C_k}\partial_y\right)=\tau_kX_{LR}|_{C_k}\,.
\end{align}
Since the gradient of $f_k$ only vanishes at the point $p_k$, and the functions $B$ and $\mu_k$ do not vanish on $C_k$ (because all the circles in $\Gamma$ are disjoint and no point $p_j$ is contained in any circle $C_k$), we infer that $X_T$ has no zeros on $C_k$, which is then a periodic orbit.

Let us assume that the periodic orbit $C_k$ is not a limit cycle. Since $X_T$ is polynomial, it then follows that $C_k$ must belong to a period annulus, i.e. it is surrounded by a family of periodic orbits. Consider a periodic orbit $\gamma_k$ close enough to $C_k$ so that it is disjoint from the set $\Gamma$ and all the points $p_j$. In particular, we have that $V_T|_{\gamma_k}$ does not vanish. Defining the $1$-form $\omega_T:=-Q_Tdx+P_Tdy$, we have that
\begin{equation}\label{eqiom}
\int_{\gamma_k}\frac{\omega_T}{V_T}=0
\end{equation}
because $\gamma_k$ is a periodic orbit of $X_T$. On the other hand, using the definitions of $P_T$, $Q_T$ and $V_T$, we can also write
\begin{align*}
\int_{\gamma_k}\frac{\omega_T}{V_T}&=-\int_{\gamma_k}dH_T=-\int_{\gamma_k}\Big[d\left(\ln A_T\right)+d\left(\ln B\right)+d\left(\ln C\right)\Big]\\&=2\sum_{j=1}^r\int_{\gamma_k}d\theta_j=\pm 4\pi s_k\neq 0\,,
        \end{align*}
where we have used that the functions $\ln A_T$ and $\ln B$ are smooth on $\gamma_k$ because $A_T$ and $B$ do not vanish there, and hence $\int_{\gamma_k}d\left(\ln A_T\right)=0=\int_{\gamma_k}d\left(\ln B\right)$. The number $s_k$ in the last equality is the number of primary cycles contained in $D_{C_k}$, which is at least $1$ (it is $1$ if and only if $C_k$ is primary). The sign of the last expression depends on how we orient the circle $C_k$. Since this formula contradicts Eq.~\eqref{eqiom}, we conclude that all the periodic orbits $C_k$ are limit cycles of $X_T$. Moreover, since the vanishing order of the inverse integrating factor $V_T$ on each cycle $C_k$ is $1$, it follows that all the limit cycles are hyperbolic~\cite{enciso peralta 2009,GGG10}.

To conclude the proof of the theorem, let us show that each constant $\tau_k$ can be chosen in such a way that the period of the limit cycle $C_k$ is $T_k$.  Indeed, since $X_T|_{C_k}=\tau_kX_{LR}|_{C_k}$ by Eq.~\eqref{eqxTck}, denoting by $T^{LR}_k$ the period of the limit cycle $C_k$ for the vector field $X_{LR}$, we have that the orbit $C_k$ has period $T_k$ if we choose
\[
\tau_k=\frac{T^{LR}_k}{T_k}\,.
\]
The theorem then follows.
\end{proof}

\section{A planar polynomial vector field with prescribed limit cycles, multiplicities and periods}\label{section multiplicity}
We first prove that, for any configuration of cycles $\Gamma$, we can construct a planar polynomial vector field that realizes it as a set of limit cycles with prescribed multiplicities. As in Section~\ref{section periods}, the method of the proof is modeled upon the construction of Llibre and Rodr\'iguez~\cite{Llibre Rodriguez 2004}.

\begin{theorem}\label{teorema multiplicity}
Let $\Gamma=\{C_1,\ldots,C_n\}$ be a configuration of cycles and $\{m_1,\ldots,m_n\}$ a set of positive integers, then $\Gamma$ is realized by a planar polynomial vector field $X_m$ with $\deg(X_m)<2(r+\sum m_k)$, where each periodic orbit $C_k$ has multiplicity $m_k$. Here $r$ is the number of primary cycles in $\Gamma$. Moreover, $X_m$ admits a polynomial inverse integrating factor, it is Darboux integrable and all its limit cycles are algebraic.
\end{theorem}
\begin{proof}
Arguing as in the proof of Theorem~\ref{teorema period}, we can assume that $\Gamma$ is a set of $n$ circles $C_k=\{f_k(x,y)=0\}$, see Eq.~\eqref{eq:circles} for the definition of $f_k$, where the first $r$ ones are the primary cycles and no point $p_k$ lies on any $C_j$.

We consider the auxiliary functions $A,B,C$ defined in Section~\ref{section periods}, as well as the functions:
\begin{multicols}{2}\noindent
\begin{align*}
    &A_m:=\prod_{k=1}^n f_k^{m_k}\,,\\
    &D:=ABC\,,\\
    &\mathcal{D}:=D\exp\left(\Lambda\sum_{k=1}^n h_k\right)\,,\\
    &H:=\ln{D}\,,\\
    &\lambda_k:=\prod_{j\neq k}^n f_j^{m_j}\,,\\
    &F:=\Lambda\sum_{k=1}^n\lambda_k\frac{\partial f_k}{\partial y}\,,\\
    &G:=\Lambda\sum_{k=1}^n\lambda_k\frac{\partial f_k}{\partial x}\,,
\end{align*}
\end{multicols}
\noindent where $\Lambda$ is a constant defined as
$$\Lambda:=\sum_{k=1}^n(m_k-1)\,,$$ 
which vanishes if and only if $m_k=1$ for every $k=1,\cdots,n$, and 
\begin{align*}
h_k:=\begin{cases}
\dfrac{f_k^{1-m_k}}{(1-m_k)}& \text{if } m_k\geq2\,,\\
\ln f_k& \text{if }m_k=1\,.
\end{cases}
\end{align*}

We claim that the vector field $X_m$ in the statement of the theorem can be defined using these functions as:
\[
X_m:=P_m\partial_x+Q_m\partial_y\qquad\text{with}\qquad\begin{array}{l}
P_m:=-A_mB\dfrac{\partial H}{\partial y}-BF\,,\\[2ex]
Q_m:=A_mB\dfrac{\partial H}{\partial x}+BG\,.
\end{array}
\]

Observe that if $m_k=1$ for all $k$, then $A_m=A$ and $F=0=G$, thus we have $X_m|_{m_k=1}=X_{LR}$, where $X_{LR}$ is the vector field constructed by Llibre and Rodr\'{i}guez in~\cite{Llibre Rodriguez 2004}.

To prove that $X_m$ satisfies the desired properties, let us show that it is a polynomial vector field with $\deg(X_m)<2(r+\sum m_k)$, that $\mathcal{D}$ is a Darboux first integral and that $V_m:=A_mB$ is an inverse integrating factor. Indeed, by direct computations we obtain that $P_m$ and $Q_m$ can be written as
\begin{align}
P_m=P_{LR}\prod_{k=1}^n f_k^{m_k-1}-BF\,,\label{eqpm}\\
Q_m=Q_{LR}\prod_{k=1}^n f_k^{m_k-1}+BG\,.\label{eqqm}
\end{align}
From these expressions it follows that $P_m$ and $Q_m$ are polynomials of degree at most $2(n+\sum m_k)-1$. Moreover, using the identities:
\[\frac{\partial\mathcal{D}}{\partial x}=\mathcal{D}\left(\frac{1}{D}\frac{\partial D}{\partial x}+\frac{G}{A_m}\right)\,,\qquad \frac{\partial\mathcal{D}}{\partial y}=\mathcal{D}\left(\frac{1}{D}\frac{\partial D}{\partial y}+\frac{F}{A_m}\right)\,,\]
one can easily check that $\mathcal{D}$ is a Darboux first integral of $X_m$. It is also straightforward to show that the vector field $\frac{X_m}{V_m}$ is divergence-free in $\mathbb R^2 \setminus V_m^{-1}(0)$, thus implying that $V_m$ is an inverse integrating factor of $X_m$.

Arguing as in the proof of Theorem~\ref{teorema period}, we conclude that the circles $\{C_1,\cdots,C_n\}$ are invariant under the flow of $X_m$, and that if $X_m$ has a limit cycle, it has to be precisely one of these circles. Let us now prove that all of them are limit cycles.

First we check that $X_m$ does not vanish on each $C_k$. Using Eqs.~\eqref{eqpm} and~\eqref{eqqm} we can write
\begin{align}\label{eqxmck}
X_m|_{C_k}=B|_{C_k}\lambda_k|_{C_k}\left(\Lambda+f_k^{m_k-1}|_{C_k}\right)\left(-\frac{\partial{f_k}}{\partial y}\Big|_{C_k}\partial_x+\frac{\partial f_k}{\partial x}\Big|_{C_k}\partial_y\right)\,.
\end{align}
Notice that the functions $B$, $\lambda_k$ and $\Lambda+f_k^{m_k-1}$ do not vanish on $C_k$ because all the circles in $\Gamma$ are disjoint, no point $p_j$ is contained in a circle $C_k$ and $\Lambda>0$ unless $m_j=1$ for all $j$. In the case that all the limit cycles have multiplicity $1$, it follows that $\Lambda+f_k^{m_k-1}=1$. Moreover, the gradient of $f_k$ only vanishes at the point $p_k$, so we infer that $X_m$ has no zeros on $C_k$, which is then a periodic orbit.

Since the vector field $X_m$ is polynomial, if $C_k$ is not a limit cycle then it must belong to a period annulus. Let us assume that this is the case, and take a periodic orbit $\gamma_k$ close enough to $C_k$ so that it is disjoint from the set $\Gamma$ and all the points $p_j$. In particular, we have that the function $V_m|_{\gamma_k}$ does not vanish. Defining the $1$-form $\omega_m:=-Q_mdx+P_mdy$, we have that
\begin{equation}\label{eqiomm}
\int_{\gamma_k}\frac{\omega_m}{V_m}=0
\end{equation}
because $\gamma_k$ is a periodic orbit of $X_m$. Using the definitions of $P_m$, $Q_m$ and $V_m$, and proceeding as in the proof of Theorem~\ref{teorema period}, we can also write
\begin{align*}
\int_{\gamma_k}\frac{\omega_m}{V_m}&=-\int_{\gamma_k}\left[dH+\Lambda\sum_{j=1}^n dh_j\right]=\\
 &=-\int_{\gamma_k}\Big[d\left(\ln A\right)+d\left(\ln B\right)+d\left(\ln C\right)\Big]=\pm 4\pi s_k\neq 0\,.
\end{align*}
To obtain this formula we have used that the functions $\ln A$, $\ln B$ and $h_j$ are smooth on $\gamma_k$ because $A$, $B$ and $f_j$ do not vanish there. The number $s_k$ in the last equality is the number of primary cycles contained in $D_{C_k}$, which is at least $1$ by definition. Since this formula contradicts Eq.~\eqref{eqiomm}, we deduce that all the periodic orbits $C_k$ are limit cycles of $X_m$.

To conclude the proof of the theorem, we notice that, by construction, the vanishing order of the inverse integrating factor $V_m$ on the limit cycle $C_k$ is $m_k$. Since the multiplicity of a limit cycle is equal to the vanishing order of the inverse integrating factor~\cite{enciso peralta 2009,GGG10}, it follows that the multiplicity of $C_k$ is $m_k$, in particular $C_k$ is hyperbolic if and only if $m_k=1$.
\end{proof}

Combining the constructions in Theorems~\ref{teorema period} and~\ref{teorema multiplicity}, and arguing exactly as in their proofs, it is easy to check that the vector field $X_{Tm}$ defined as
\[
X_{Tm}:=P_{Tm}\partial_x+Q_{Tm}\partial_y\,,\qquad\text{with}\qquad\begin{array}{l}
P_{Tm}:=-A_mB\dfrac{\partial H_T}{\partial y}-BF_T\,,\\[2ex]
Q_{Tm}:=A_mB\dfrac{\partial H_T}{\partial x}+BG_T\,,
\end{array}
\]
realizes the set of cycles $\Gamma$ with prescribed periods and multiplicities. Here the functions $B$ and $H_T$ were defined in the proof of Theorem~\ref{teorema period}, $A_m$ in the proof of Theorem~\ref{teorema multiplicity} and we set 
\begin{align*}
&F_T:=\Lambda\sum_{k=1}^n\tau_k\lambda_k\frac{\partial f_k}{\partial y}\,,\\
&G_T:=\Lambda\sum_{k=1}^n\tau_k\lambda_k\frac{\partial f_k}{\partial x}\,.
\end{align*}
The constants $\tau_k$ are chosen so that the limit cycle $C_k$ has period $T_k$. Indeed, noticing that
\begin{equation}\label{eq:xtm}
X_{Tm}|_{C_k}=\tau_k\Big(\Lambda+f_k^{m_k-1}|_{C_k}\Big)\prod_{j\neq k}f_j^{m_j-1}|_{C_k}X_{LR}|_{C_k}\,,
\end{equation}
if we denote by $\gamma^{LR}_k(t)$ the integral curve parametrizing $C_k$ of the Llibre-Rodriguez vector field $X_{LR}$ realizing $\Gamma$, the constant $\tau_k$ is chosen as
$$
\tau_k=\frac{1}{T_k}\int_0^{T_k^{LR}}\frac{dt}{[\Lambda+f_k^{m_k-1}(\gamma_k^{LR}(t))]\prod_{j\neq k}f_j^{m_j-1}(\gamma_k^{LR}(t))}\,,
$$
thus implying that the period of $C_k$ for the vector field $X_{Tm}$ is $T_k$. Here $T_k^{LR}$ is the period of the integral curve $\gamma^{LR}_k(t)$. The theorem can then be stated as follows:

\begin{theorem}\label{teorema multiplicity and period}
Let $\Gamma=\{C_1,\ldots,C_n\}$ be a configuration of cycles, $\{m_1,\ldots,m_n\}$ a set of positive integers and $\{T_1,\ldots,T_n\}$ a set of positive constants, then $\Gamma$ is realized by a planar polynomial vector field $X_{Tm}$ with $\deg(X_{Tm})<2(r+\sum m_k)$, where each periodic orbit $C_k$ has multiplicity $m_k$ and period $T_k$. Here $r$ is the number of primary cycles in $\Gamma$. Moreover, $X_{Tm}$ admits a polynomial inverse integrating factor, it is Darboux integrable and all its limit cycles are algebraic.
\end{theorem}

\section{A planar polynomial vector field with prescribed limit cycles, periods and stabilities}\label{section stability}

The goal of this section is to prove that, for any configuration of cycles $\Gamma$, we can construct a planar polynomial vector field that realizes it as a set of hyperbolic limit cycles with prescribed stabilities and periods. To this end, we first show that the stability of each limit cycle $C_k$ of the polynomial vector field $X_T$ constructed in Section~\ref{section periods} can be characterized in terms of the relative position of $C_k$ with respect to the other cycles. In what follows, we will say that the limit cycle $C_k$ has stability $-1$ if it is stable and $1$ if it is unstable. Since all the limit cycles considered in this section are hyperbolic, the interior and exterior stabilities are the same; the case of semistable limit cycles will be addressed in Section~\ref{section main theorem}.

\begin{lemma}\label{lemma stability m=1}
The limit cycle $C_k$ of the vector field $X_T$ introduced in Section~\ref{section periods} has stability $(-1)^{N_k}$, where $N_k:=\text{card}\,\{C_j: C_k\subset D_{C_j}, j\neq k\}$.  
\end{lemma}
\begin{proof}
Denoting by $\gamma_k(t)$, $t\in\R$, the integral curve of $X_T$ whose image is the limit cycle $C_k$, we first compute the derivative of the angular variable $\theta_j$ defined in Eq.~\eqref{thetak} for each $j$: 
\begin{align}
      \frac{d\theta_j(\gamma_k(t))}{dt}&=X_T(\gamma_k(t))\cdot \nabla\theta_j(\gamma_k(t))\nonumber\\&=2\tau_k B(\gamma_k(t))\mu_k(\gamma_k(t))\,\frac{(x-x_k)(x-x_j)+(y-y_k)(y-y_j)}{(x-x_j)^2+(y-y_j)^2}\Big|_{\gamma_k(t)}\,,
   \end{align}
which is negative (positive) if the orientation of $C_k$ induced by the integral curve $\gamma_k(t)$ is clockwise (counterclockwise). The constants $\tau_k$ and the functions $B$ and $\mu_k$ were defined in the proof of Theorem~\ref{teorema period}. Assuming that the point $(x_j,y_j)$ is contained in the interior of the compact set $D_{C_k}$, it easily follows that
\begin{equation}\label{eqconvex}
\frac{(x-x_k)(x-x_j)+(y-y_k)(y-y_j)}{(x-x_j)^2+(y-y_j)^2}\Big|_{\gamma_k(t)}>0\,.
\end{equation}
Accordingly, since $B$ is always positive over $C_k$, the sign of $\frac{d\theta_j(\gamma_k(t))}{dt}$ is given by the sign of $\mu_k$ over $C_k$. Noticing that the function $f_j=(x-x_j)^2+(y-y_j)^2-r_j^2$ is positive over $C_k$ if and only if $C_k$ is not contained in $D_{C_j}$, the definition of $\mu_k$ implies that its sign is precisely $(-1)^{N_k}$.

The stability of the limit cycle $C_k$ is given by the sign of the following integral~\cite{dumortier2006qualitative}
$$
\mathcal L:=\int_0^{T_k}\Div X_T(\gamma_k(t))dt\,.
$$ 
Using the identity $X_T\cdot{}\nabla V_T=V_T\Div X_T$ we obtain
    \begin{align*}
\mathcal L&=\int_{0}^{T_k}X_T(\gamma_k(t))\cdot\nabla\ln A(\gamma_k(t))dt\\&+\int_{0}^{T_k}X_T(\gamma_k(t))\cdot\nabla\ln B(\gamma_k(t))dt\\&=\int_{0}^{T_k}X_T(\gamma_k(t))\cdot\nabla\ln f_k(\gamma_k(t))dt\,,
    \end{align*}
where to pass to the second equality we have used that the functions $\ln B$ and $\ln f_j$ with $j\neq k$ are smooth on $C_k$, and therefore the corresponding integrals vanish. Using the definition of $X_T$, the identities~\eqref{idC} and the value of $X_T$ on $C_k$ computed in~\eqref{eqxTck}, after a few straightforward computations we can write
\begin{align*}
\mathcal L&= \frac{-1}{\tau_k}\int_{0}^{T_k}X_T(\gamma_k(t))\cdot \nabla \ln C(\gamma_k(t))\\&+\frac{-1}{\tau_k}\int_{0}^{T_k}X_T(\gamma_k(t))\cdot \nabla \ln \Big(B\prod_{j\neq k}f_j^{\tau_j}\Big)(\gamma_k(t))\,.
\end{align*}
As before, the second integral in this expression vanishes because the function $B\prod_{j\neq k}f_j^{\tau_j}$ is smooth on $C_k$. Finally, from the definition of the function $C$ we obtain
$$
\mathcal L=\frac{2}{\tau_k}\sum_{j=1}^r\int_0^{T_k}\frac{d\theta_j(\gamma_k(t))}{dt}dt=\frac{(-1)^{N_k}4\pi s_k}{\tau_k}\,,
$$
where $s_k\in\{1,\ldots,r\}$ is the number of primary cycles contained in $D_{C_k}$, and we have used the sign $(-1)^{N_k}$ computed before (observe that there is a non vanishing contribution to the integral above if and only if $\theta_j$ is the angle whose center $(x_j,y_j)$ is contained in $D_{C_k}$, so that Eq.~\eqref{eqconvex} holds). We then conclude that the stability of $C_k$ is $(-1)^{N_k}$, as we wanted to show.
\end{proof}

This lemma proves that the stability of the limit cycle $C_k$ of $X_T$ is fixed by the configuration of cycles that we want to realize. In the following theorem, which is the main result of this section, we show how to modify the vector field $X_T$ in order to prescribe the stabilities of its limit cycles. The idea is to add additional cycles to the configuration $\Gamma$ to obtain a new configuration $\tilde\Gamma$ so that $\tilde N_k$ for the new configuration has the desired sign. Since we do not want to realize the extra cycles $\tilde\Gamma\backslash\Gamma$, we can remove them by adding a singular (zero) point over each extra limit cycle of the vector field $\tilde X_T$ realizing $\tilde \Gamma$.

\begin{theorem}\label{teorema stability}
Let $\Gamma=\{C_1,\ldots,C_n\}$ be a configuration of cycles, $\{T_1,\ldots,T_n\}$ a set of positive constants and $\{\nu_1,\ldots,\nu_n\}$ a set of $\pm1$. Then $\Gamma$ is realized by a planar polynomial vector field  $X_{Ts}$ with $\deg(X_{Ts})<2(3n+r)$, where each periodic orbit $C_k$ is hyperbolic, has period $T_k$ and stability $\nu_k$. Here $r$ is the number of primary cycles in $\Gamma$. Moreover, $X_{Ts}$ admits a polynomial inverse integrating factor, it is Darboux integrable and all its limit cycles are algebraic.
\end{theorem}
\begin{proof}
As in previous sections, we can assume that $\Gamma$ consists of circles. Take a circle $C_{n+k}$ centered at $p_k$ of radius $r_k-\varepsilon\nu_k$. Recall that $r_k$ is the radius of the circle $C_k$ and $p_k$ is its center. It is clear that we can take $\varepsilon>0$ small enough such that all the circles are disjoint and no $p_j$ lies on any $C_k$ and $C_{n+k}$. We denote the whole configuration by $\tilde{\Gamma}:=\{C_1,\cdots,C_{2n}\}$. Observe that the number of primary cycles and their centers remain unchanged (a primary cycle $C_k$ in $\Gamma$ with $\nu_k>0$ is no longer a primary cycle in $\tilde{\Gamma}$, instead $C_{n+k}$ will be primary, but it has the same center), and therefore the function $B$ will be the same for $\Gamma$ and $\tilde{\Gamma}$.

Now we construct a vector field $\tilde{X}_T$ as in Theorem~\ref{teorema period} realizing the $2n$ cycles of $\tilde{\Gamma}$ where each limit cycle $C_k$ has an associated constant $\tau_k$ that will be fixed later (cf. the proof of Theorem~\ref{teorema period} for the definition of such a constant) and $\tau_{n+k}=1$ so that the limit cycle $C_{n+k}$ has period $T_{n+k}^{LR}$, for $k\in\{1,\ldots,n\}$. It is obvious from the definition of $\tilde\Gamma$ that the parity of the number $\tilde N_k$ defined in Lemma~\ref{lemma stability m=1} only depends on the relative positions of $C_{n+k}$ and $C_k$, and that $(-1)^{\tilde N_k}=\nu_k$ for $k=1,\ldots,n$. The lemma then implies that the cycles $C_k$ have the desired stability. In order to remove the additional cycles $C_{n+k}$, we consider the functions:
\begin{align*}
  &l_k(x,y):=(x-a_k)^2+(y-b_k)^2\\
  &L_{Ts}:=\prod_{k=1}^n l_k
\end{align*}
where each $q_k:=(a_k,b_k)\in C_{n+k}$ is a point at the extra cycle $C_{n+k}$. Then the vector field:
\[X_{Ts}:=L_{Ts} \tilde{X}_T\]
satisfies the statements of the theorem. Indeed, since the factor $L_{Ts}$ is positive on $\Gamma$, this set is realized by $X_{Ts}$ as algebraic limit cycles, while the cycles $C_{n+k}$ contain a singular point of $X_{Ts}$, thus becoming homoclinic connections. The factor $L_{Ts}$ does not change the stability of each cycle $C_k$ of $\tilde X_T$, which is $\nu_k$. Moreover, if $\gamma^{LR}_k(t)$ is the integral curve parametrizing $C_k$ of the Llibre-Rodriguez vector field $X_{LR}$ that realizes $\tilde \Gamma$, see Section~\ref{section periods} for the definition, the constant $\tau_k$ must be chosen such that
$$
\tau_k=\frac{1}{T_k}\int_0^{T_k^{LR}}\frac{dt}{L_{Ts}(\gamma^{LR}_k(t))}\,,
$$
which implies that the period of $C_k$ for the vector field $X_{Ts}$ is $T_k$. Here $T_k^{LR}$ is the period of the integral curve $\gamma^{LR}_k(t)$.
Finally, observe that the degree of $X_{Ts}$ is $2(3n+r)$ because the vector field $\tilde X_T$ has degree $2(2n+r)$ and the factor $L_{Te}$ has degree $2n$. Additionally, $V_{Ts}:=ABL_{Ts}$ is an inverse integrating factor of $X_{Ts}$, and the Darboux first integral $\tilde H_T$ of $\tilde X_T$ is a first integral of $X_{Ts}$ as well. This completes the proof of the theorem.
\end{proof}

\section{The main theorem}\label{section main theorem}
In this section we construct a vector field $X$ that realizes a configuration of cycles $\Gamma$ with prescribed periods, stabilities and multiplicities, thus establishing the main theorem of the paper. As in Section~\ref{section stability} we shall use $\nu_k\in\{-1,1\}$ to denote the \emph{interior stability} that we want to prescribe for the limit cycle $C_k$ (negative means stable in the interior and positive is unstable). Observe that the \emph{exterior stability} of $C_k$ is determined by $\nu_k$ and the multiplicity $m_k$ as $(-1)^{m_k+1}\nu_k$. In the following lemma we show that the interior stability of each limit cycle $C_k$ of the polynomial vector field $X_{Tm}$ constructed in Section~\ref{section multiplicity} can be characterized in terms of the relative position of $C_k$ with respect to the other limit cycles and the set of multiplicities $\{m_1,\ldots,m_n\}$. In the case that $m_k=1$ for all $k$ we recover Lemma~\ref{lemma stability m=1}.

\begin{lemma}\label{lemma stability m>1}
The limit cycle $C_k$ of the vector field $X_{Tm}$ constructed in Section~\ref{section multiplicity} has interior stability $\nu_k=(-1)^{m_k+M_k+1}$, where
\[M_k:=\sum_{j\in\mathcal{M}_k}m_j\]
and $\mathcal{M}_k:=\{j\in\{1,\ldots,n\}\ /\ C_k\subset D_{C_j}, j\neq k\}$. 
\end{lemma}
\begin{proof}
Consider a closed curve $\hat C_k$ in the region bounded by $C_k$ and close enough to it so that $\hat C_k$ is disjoint from all $C_j$ and $p_j$. Since $C_k$ is a limit cycle of $X_{Tm}$ we can assume that $\hat C_k$ is transverse to the integral curves of $X_{Tm}$ at each point. The interior stability of $C_k$ is determined by the sign of the flux through $\hat C_k$ of $X_{Tm}$, which is given by 
\begin{equation}\label{flux}
Flux_k:=\int_{\hat C_k}X_{Tm}\cdot n ds\,,
\end{equation}
where $n$ is the unit normal vector on $\hat C_k$ pointing outwards and $s$ parametrizes the curve $\hat C_k$ in the positive direction (i.e. counterclockwise).

The sign of the flux~\eqref{flux} can be easily computed using the $1$-form $\omega_{Tm}:=-Q_{Tm}dx+P_{Tm}dy$ and the inverse integrating factor $V_{Tm}:=A_mB$ of $X_{Tm}$. Indeed, noticing that $V_{Tm}$ does not vanish at any point of $\hat C_k$, and observing that 
$$
\frac{\omega_{Tm}}{V_{Tm}}=-dH_T-\Lambda\sum_{j=1}^n\frac{df_j}{f_j^{m_j}}\,,
$$
see the definitions of all the involved functions in Sections~\ref{section multiplicity} and~\ref{section stability}, it is obvious that the sign of $Flux_k$ is the same as the sign of
\begin{align*}
\widehat{Flux_k}:=V_{Tm}(p_0)\int_{\hat C_k}\frac{\omega_{Tm}}{V_{Tm}}= -V_{Tm}(p_0)\int_{\hat C_k}dH_T=4\pi s_k V_{Tm}(p_0)\,,
\end{align*}
where $p_0$ is any fixed point on $\hat C_k$. The second equality follows from the fact that each term $\frac{df_j}{f_j^{m_j}}$ does not contribute to the integral because $f_j$ does not vanish at any point of $\hat C_k$. For the last equality we have used the expression of $\int dH_T$ obtained in the proof of Theorem~\ref{teorema period}, taking into account that the curve $\hat C_k$ is positively oriented. 
  
Since the interior stability of $C_k$ is given by $-\sign(Flux_k)$, the formula above implies that $\nu_k=-\sign(V_{Tm}(\hat C_k))=-\sign(A_m(\hat C_k))$. Then, arguing as in the proof of Lemma~\ref{lemma stability m=1}, but taking into account the multiplicity, we easily obtain the desired expression for $\nu_k$.
\end{proof}

In the following theorem we show how to modify the vector field $X_{Tm}$ in order to prescribe the interior stabilities of its limit cycles. The idea is the same as in the proof of Theorem~\ref{teorema stability}: we add additional cycles to the configuration $\Gamma$ to obtain a new configuration $\tilde\Gamma$ so that $\tilde \nu_k$ for the new configuration has the desired sign. To remove the extra cycles $\tilde\Gamma\backslash\Gamma$, we add a singular point over each extra limit cycle of the vector field $\tilde X_{Tm}$ realizing $\tilde \Gamma$.

  \begin{theorem}\label{main theorem}
    Let $\Gamma=\{C_1,\ldots,C_n\}$ be a configuration of cycles, $\{m_1,\ldots,m_n\}$ a set of positive integers, $\{T_1,\ldots,T_n\}$ a set of positive constants and $\{\nu_1,\ldots,\nu_n\}$ a set of $\pm1$. Then $\Gamma$ is realized by a planar polynomial vector field $X$ with
    \[\deg(X)<2\left(2(N-n)+r+\sum_{k=1}^n m_k\right)\]
    where each periodic orbit $C_k$ has multiplicity $m_k$, period $T_k$ and interior stability $\nu_k$. As usual, $r$ is the number of primary cycles in $\Gamma$, and $N:=n+n_1+2n_2$ where $n_1:=card\,\{k\ :\ m_k \text{ is odd}\}$ and $n_2:=card\,\{k\ :\ m_k \text{ is even and } M_k \text{ is odd}\}$. Moreover, $X$ admits a polynomial inverse integrating factor, it is Darboux integrable and all its limit cycles are algebraic.
\end{theorem}
\begin{proof}
As usual, we assume that $\Gamma$ consists of circles. Let us define a new configuration of circles $\tilde{\Gamma}:=\Gamma\cup\{C_{n+1},\ldots,C_{N}\}$ following these rules:
\begin{itemize}
       \item If $m_k$ is odd, we add a concentric circle of radius $r_k-\nu_k\varepsilon$.
       \item If $m_k$ is even and $\nu_k=-1$, we add nothing.
       \item If $m_k$ is even and $\nu_k=1$, we add two concentric circles of radii $r_k\pm\varepsilon$.
\end{itemize}
Here $\varepsilon$ is a small enough constant so that all the circles in $\tilde \Gamma$ are disjoint, and disjoint from $p_k$. Using Lemma~\ref{lemma stability m>1} it is easy to check  that $N=n+n_1+2n_2$.
 
Now, we construct a vector field $\tilde X_{Tm}$ as in Section~\ref{section multiplicity} that realizes the configuration $\tilde\Gamma$, with $m_k$, $k\in\{1,\ldots,n\}$, the multiplicity we want to prescribe and $m_{n+j}=1$, $j\in\{1,\ldots,N-n\}$. Moreover, the constant $\tau_k$ corresponding to each limit cycle $C_k$ will be fixed later, while we take $\tau_{n+j}=1$. A simple argument using Lemma~\ref{lemma stability m>1} implies that the interior stability of each limit cycle $C_k$ of $\tilde X_{Tm}$ is precisely $\nu_k$. 

Taking an arbitrary point $q_j:=(a_j,b_j)\in C_{n+j}$ for every extra circle, we define the function $L$ as in the proof of Theorem~\ref{teorema stability}, i.e.
\[L:=\prod_{j=1}^{N-n} l_j\,,\]    
where $l_j:=(x-a_j)^2+(y-b_j)^2$, and the vector field
$$
X:=L\tilde X_{Tm}\,.
$$

Since the factor $L$ is positive on $\Gamma$, this set is realized by $X$ as algebraic limit cycles, while the remaining cycles $C_{n+j}\subset\tilde\Gamma$ contain a singular point, thus becoming homoclinic connections. The factor $L$ does not change the interior stability of each cycle $C_k$, which is then $\nu_k$ by construction, nor the multiplicity $m_k$. The constants $\{\tau_k\}_{k=1}^n$ can also be chosen such that each limit cycle $C_k$ of $X$ has period $T_k$. More precisely, since the vector field $\tilde X_{Tm}$ on each limit cycle $C_k$ can be written as in Eq.~\eqref{eq:xtm}, we conclude that 
$$
\tau_k=\frac{1}{T_k}\int_0^{T_k^{LR}}\frac{dt}{L(\gamma_k^{LR}(t))[\tilde\Lambda+f_k^{m_k-1}(\gamma_k^{LR}(t))]\prod_{j\neq k}^Nf_j^{m_j-1}(\gamma_k^{LR}(t))}\,,
$$
where we are using the notation introduced in Section~\ref{section multiplicity}.

Finally, an easy computation shows that $X$ is a polynomial vector field of degree as in the statement of the theorem, and $V:=A_mBL$ is an inverse integrating factor. Using the functions defined in Sections~\ref{section periods} and~\ref{section multiplicity}, it is ready to check that the function
$$
A_TBC\exp{\Big(\Lambda \sum_{j=1}^N\tau_jh_j\Big)}
$$ 
is a Darboux first integral of the vector field $X$. This completes the proof of the theorem.
\end{proof}

\begin{remark}
We observe that the vector field $X$ in the proof of Theorem~\ref{main theorem} can be constructed without including the functions $f_{n+j}$, $j\in\{1,\ldots,N-n\}$, in the quantities $H_T$, $F_T$ an $G_T$ appearing in the definition of $\tilde X_{Tm}$.
\end{remark}

\section*{Acknowledgments}
The authors are supported by a ``La Caixa'' fellowship (J.M.-B.) and the ERC Starting Grant 335079 (D.P.-S.). This work is supported in part by the ICMAT--Severo Ochoa grant SEV-2011-0087.

\bibliographystyle{amsplain}

\end{document}